\documentclass{amsart}

\usepackage{amsfonts}
\usepackage{ifthen}
\usepackage{amsthm}
\usepackage{amsmath}
\usepackage{graphicx}
\usepackage{amscd,amssymb,amsthm}
\usepackage{graphicx}

\newtheorem{THEO}{\bf Theorem}

\newcommand{\lp}{\mathcal{LP}}
\newcommand{\lpi}{\mathcal{LP}I}
\newcommand{\g}{\gamma}

\newcounter{minutes}
\divide\time by 60
\newcounter{hours}
\multiply\time by 60 \addtocounter{minutes}{-\time}

\newtheorem{lemma}{Lemma}
\newtheorem{theorem}{Theorem}

\newcommand{\real}{\operatorname{Re}}


\begin{document}
\title[Radii of starlikeness of some special functions]{Radii of
starlikeness of some special functions}

\author[\'A. Baricz]{\'Arp\'ad Baricz}
\address{Department of Economics, Babe\c{s}-Bolyai University, Cluj-Napoca
400591, Romania} \email{bariczocsi@yahoo.com}

\author[D. K. Dimitrov]{Dimitar K. Dimitrov}
\address{Departamento de Matem\'atica Aplicada, IBILCE, Universidade Estadual Paulista UNESP, S\~{a}o Jos\'e do Rio Preto 15054, Brazil}
\email{dimitrov@ibilce.unesp.br}

\author[H. Orhan]{Halit Orhan}
\address{Department of Mathematics, Ataturk University, Erzurum 25240, Turkey}
\email{horhan@atauni.edu.tr}

\author[N. Yagmur]{Nihat Yagmur}
\address{Department of Mathematics, Erzincan University, Erzincan 24000,
Turkey} \email{nhtyagmur@gmail.com}

\thanks{The research of \'A. Baricz is supported by the Romanian National Authority for Scientific Research, CNCS-UEFISCDI, under Grant
PN-II-RU-TE-2012-3-0190. The research of D. K. Dimitrov is supported by the Brazilian foundations CNPq under Grant 307183/2013--0
and FAPESP under Grants 2009/13832--9}


\begin{abstract}
Geometric properties of the classical Lommel and Struve functions, both of the first kind, are studied.
For each of them, there different normalizations are applied in such a way that
the resulting functions are analytic in the unit disc of the complex plane.
For each of the six functions we determine the radius of starlikeness precisely.
\end{abstract}

\maketitle



\section{Introduction and statement of the main results}

Let $\mathbb{D}_{r}$ be the open disk $\left\{ {z\in \mathbb{C}:\left\vert
z\right\vert <r}\right\} ,$ where $r>0,$ and set $\mathbb{D}=\mathbb{D}_{1}$.
By $\mathcal{A}$ we mean the class of analytic
functions $f:\mathbb{D}_r\to\mathbb{C}$ which satisfy the usual
normalization conditions $f(0)=f'(0)-1=0.$ Denote by $\mathcal{S}$
the class of functions belonging to $\mathcal{A}$ which are univalent in $\mathbb{D}%
_r$ and let $\mathcal{S}^{\ast }(\alpha )$ be the subclass of $\mathcal{S}$
consisting of functions which are starlike of order $\alpha $ in $\mathbb{D}%
_r,$ where $0\leq \alpha <1.$ The analytic characterization of this class of
functions is
\begin{equation*}
\mathcal{S}^{\ast }(\alpha )=\left\{ f\in \mathcal{S}\ :\  \Re \left(\frac{zf'(z)}{f(z)}\right)>\alpha\ \ \mathrm{for\ all}\ \ z\in \mathbb{%
D}_r \right\},
\end{equation*}
and we adopt the convention $\mathcal{S}^{\ast}=\mathcal{S}^{\ast }(0)$. The real number
\begin{equation*}
r_{\alpha }^{\ast}(f)=\sup \left\{ r>0\ :\ \Re \left(\frac{zf'(z)}{f(z)}\right)>\alpha\ \ \mathrm{ for\ all}\ \ z\in \mathbb{D}_{r}\right\},
\end{equation*}
is called the radius of starlikeness of order $\alpha $ of the function $f.$
Note that $r^{\ast }(f)=r_0^{\ast}(f)$ is the largest radius such that the
image region $f(\mathbb{D}_{r^{\ast }(f)})$ is a starlike domain with
respect to the origin.

We consider two classical special functions, the Lommel function of the first kind $s_{\mu ,\nu }$
and the Struve function of the first kind $\mathbf{H}_{\nu}$. They are explicitly defined in terms of the
hypergeometric function $\,_{1}F_{2}$ by
\begin{equation}
s_{\mu ,\nu }(z)=\frac{z^{\mu +1}}{(\mu -\nu +1)(\mu +\nu +1)}\, _{1}F_{2}\left( 1;\frac{\mu -\nu +3}{2},\frac{\mu +\nu +3}{2};-\frac{z^{2}}{4}\right),\ \ \frac{1}{2}(-\mu \pm \nu-3) \not\in \mathbb{N},
\label{LomHypG}
\end{equation}%
and
\begin{equation}
\mathbf{H}_{\nu}(z)=\frac{\left(\frac{z}{2}\right)^{\nu+1}}{\sqrt{\frac{\pi}{4}}\, \Gamma\left(\nu+\frac{3}{2}\right)}
\,_{1}F_{2} \left( 1;\frac{3}{2},\nu + \frac{3}{2};-\frac{z^{2}}{4}\right),\ \ -\nu-\frac{3}{2} \not\in \mathbb{N}.
\label{SrtHypG}
\end{equation}%
A common feature of these functions is that they are solutions of inhomogeneous Bessel
differential equations \cite{Wat}. Indeed, the Lommel function of the first kind $s_{\mu ,\nu }$ is
a solution of
\begin{equation*}
z^{2}w''(z)+zw'(z)+(z^{2}-{\nu }^{2})w(z)=z^{\mu +1}
\end{equation*}
while the Struve function $\mathbf{H}_{\nu}$ obeys
\begin{equation*}
z^{2}w''(z)+zw'(z)+(z^{2}-{\nu }^{2})w(z)=\frac{4\left(
\frac{z}{2}\right) ^{\nu +1}}{\sqrt{\pi }\Gamma \left( \nu +\frac{1}{2}%
\right) }.
\end{equation*}%
We refer to Watson's treatise \cite{Wat} for comprehensive information about these functions and recall
some more recent contributions. In 1972 Steinig \cite{stein}
examined the sign of $s_{\mu ,\nu }(z)$ for real $\mu ,\nu $ and positive $z$. He showed, among other things,
that for $\mu <\frac{1}{2}$ the function $s_{\mu ,\nu }$ has infinitely many changes of sign on $(0,\infty )$. In 2012
Koumandos and Lamprecht \cite{Kou} obtained sharp estimates for the location of the zeros of $s_{\mu -\frac{1}{2},\frac{1}{2}}$
when $\mu \in (0,1)$. The Tur\'{a}n type inequalities for $s_{\mu -\frac{1}{2},\frac{1}{2}}$  were established in \cite{Bar2} while
those for the Struve function were proved in \cite{BPS}.

Geometric properties of $s_{\mu -\frac{1}{2},\frac{1}{2}}$ and of the Struve function were obtained in \cite{Bar3} and in \cite{H-Ny,N-H},
respectively. Motivated by those results we study the problem of starlikeness of certain analytic functions related to the classical special
functions under discussion. Since neither $s_{\mu ,\nu }$, nor $\mathbf{H}_{\nu}$ belongs to $\mathcal{A}$, first we perform some natural
normalizations. We define three functions originating from $s_{\mu ,\nu }$:
\begin{equation*}
f_{\mu ,\nu }(z)=\left( (\mu -\nu +1)(\mu +\nu +1)s_{\mu ,\nu }(z)\right)^{\frac{1}{\mu +1}},
\end{equation*}%
\begin{equation*}
g_{\mu ,\nu }(z)=(\mu -\nu +1)(\mu +\nu +1)z^{-\mu }s_{\mu ,\nu }(z)
\end{equation*}
and
\begin{equation*}
h_{\mu ,\nu }(z)=(\mu -\nu +1)(\mu +\nu +1)z^{\frac{1-\mu }{2}}s_{\mu ,\nu }(%
\sqrt{z}).
\end{equation*}%
Similarly, we associate with $\mathbf{H}_{\nu}$ the functions
$$
u_{\nu }(z)=\left(\sqrt{\pi }2^{\nu }\Gamma \left( \nu +\frac{3}{2} \right)
\mathbf{H}_{\nu }(z)\right)^{\frac{1}{\nu +1}},$$
$$
v_{\nu }(z)=\sqrt{\pi }2^{\nu }z^{-\nu }\Gamma \left( \nu + \frac{3}{2} \right) \mathbf{H}_{\nu }(z)
$$
and
$$
w_{\nu }(z)=\sqrt{\pi }2^{\nu }z^{\frac{1-\nu }{2}}\Gamma \left( \nu +\frac{3}{2}\right) \mathbf{H}_{\nu }(\sqrt{z}).
$$

Clearly the functions $f_{\mu ,\nu }$, $g_{\mu ,\nu }$, $h_{\mu ,\nu }$, $u_{\nu }$, $v_{\nu }$ and $w_{\nu }$
belong to the class $\mathcal{A}$. The main results in the present note concern the exact values of the radii
of starlikeness for these six function, for some ranges of the parameters.

Let us set
$$
f_{\mu }(z)=f_{\mu -\frac{1}{2},\frac{1}{2}}(z),\ \  g_{\mu }(z)=g_{\mu-\frac{1}{2},\frac{1}{2}}(z)\ \ \ \mbox{and}\ \ \ h_{\mu }(z)=h_{\mu-\frac{1}{2},\frac{1}{2}}(z).$$
The first principal result we establish reads as follows:

\begin{theorem}
\label{theo1} Let $\mu\in(-1,1),$ $\mu\neq0.$ The following statements hold:

\begin{enumerate}
\item[\textbf{a)}] If $0\leq\alpha<1$ and $\mu\in \left(-\frac{1}{2},0\right),$ then $r_{\alpha }^{\ast }(f_{\mu })=x_{\mu
,\alpha }$, where $x_{\mu ,\alpha }$ is the smallest positive root of the
equation
\begin{equation*}
z\, s_{\mu -\frac{1}{2},\frac{1}{2}}'(z)-\alpha \left(\mu + \frac{1}{2} \right)s_{\mu -\frac{1}{2},\frac{1}{2}}(z)=0.
\end{equation*}
Moreover, if $0\leq\alpha<1$ and $\mu\in \left(-1,-\frac{1}{2}\right),$ then $r_{\alpha
}^{\ast }(f_{\mu })=q_{\mu ,\alpha }$, where $q_{\mu ,\alpha }$ is
the unique positive root of the equation $$izs_{\mu -\frac{1}{2},\frac{1}{2}%
}'(iz)-\alpha \left(\mu +\frac{1}{2}\right)s_{\mu -\frac{1}{2},\frac{1}{2}%
}(iz)=0.$$

\item[\textbf{b)}] If $0\leq\alpha<1,$ then $r_{\alpha }^{\ast }(g_{\mu })=y_{\mu
,\alpha }$, where $y_{\mu ,\alpha }$ is the smallest positive root of the
equation
\begin{equation*}
z\, s_{\mu -\frac{1}{2},\frac{1}{2}}'(z)-\left(\mu +\alpha- \frac{1}{2} \right) s_{\mu -\frac{1}{2},\frac{1}{2}}(z)=0.
\end{equation*}

\item[\textbf{c)}] If $0\leq\alpha<1,$ then $r_{\alpha }^{\ast }(h_{\mu })=t_{\mu ,\alpha }$, where $%
t_{\mu ,\alpha }$ is the smallest positive root of the equation
\begin{equation*}
zs_{\mu -\frac{1}{2},\frac{1}{2}}'(z)-\left(\mu +2\alpha -\frac{3}{2%
}\right)s_{\mu -\frac{1}{2},\frac{1}{2}}(z)=0.
\end{equation*}
\end{enumerate}
\end{theorem}

The corresponding result about the radii of starlikeness of the functions, related to Struve's one,
is:

\begin{theorem}
\label{theo2} Let $|\nu|<\frac{1}{2}.$ The following assertions are true:

\begin{enumerate}
\item[\textbf{a)}] If $0\leq
\alpha <1,$ then $r_{\alpha }^{\ast }(u_{\nu })=\delta _{\nu ,\alpha }$,
where $\delta _{\nu ,\alpha }$ is the smallest positive root of the equation
\begin{equation*}
z\mathbf{H}_{\nu }'(z)-\alpha (\nu +1)\mathbf{H}_{\nu }(z)=0.
\end{equation*}%

\item[\textbf{b)}] If $0\leq
\alpha <1,$ then $r^{\ast }(v_{\nu })=\rho _{\nu ,\alpha }$, where $\rho
_{\nu ,\alpha }$ is the smallest positive root of the equation
\begin{equation*}
z\mathbf{H}_{\nu }'(z)-(\alpha +\nu )\mathbf{H}_{\nu }(z)=0.
\end{equation*}%

\item[\textbf{c)}] If $0\leq\alpha<1,$ then $r_{\alpha }^{\ast }(w_{\nu })=\sigma _{\nu ,\alpha }$,
where $\sigma _{\nu ,\alpha }$ is the smallest positive root of the equation
\begin{equation*}
z\mathbf{H}_{\nu }'(z)-(2\alpha +\nu -1)\mathbf{H}_{\nu }(z)=0.
\end{equation*}%
\end{enumerate}
\end{theorem}

It is worth mentioning that the starlikeness of $h_{\mu }$, when $\mu \in (-1,1)$,  $\mu
\neq 0,$ as well as of $w_{\nu }$, under the restriction $\left\vert \nu \right\vert
\leq \frac{1}{2}$, were established in \cite{Bar3}, and it was proved there that all the derivatives of these
functions are close-to-convex in $\mathbb{D}.$

\section{Preliminaries}
\setcounter{equation}{0}

\subsection{The Hadamard's factorization} The following preliminary result is the content of Lemmas 1 and 2 in \cite{Bar2}.

\begin{lemma}
\label{lem1} Let
\begin{equation*}
\varphi _{k}(z)=\, _{1}F_{2}\left( 1;\frac{\mu -k+2}{2},\frac{\mu -k+3}{2};-%
\frac{z^{2}}{4}\right)
\end{equation*}%
where ${z\in \mathbb{C}}$, ${\mu \in \mathbb{R}}$ and ${k\in }\left\{
0,1,\dots\right\} $ such that ${\mu -k}$ is not in $\left\{
0,-1,\dots\right\} $. Then, $\varphi _{k}$ is an entire function of order $%
\rho =1$ and of exponential type $\tau =1.$ Consequently, the Hadamard's
factorization of $\varphi _{k}$ is of the form%
\begin{equation}
\varphi _{k}(z)=\prod\limits_{n\geq 1}\left( 1-\frac{z^{2}}{z_{\mu ,k,n}^{2}}%
\right) ,  \label{1.6}
\end{equation}%
where $\pm z_{\mu ,k,1},$ $\pm z_{\mu ,k,2},\dots$ are all zeros of the
function $\varphi _{k}$ and the infinite product is absolutely convergent.
Moreover, for $z,$ ${\mu }$ and $k$ as above, we have
\begin{equation*}
(\mu -k+1)\varphi _{k+1}(z)=(\mu -k+1)\varphi _{k}(z)+z\varphi _{k}'(z),  \label{1.7}
\end{equation*}
\begin{equation*}
\sqrt{z}s_{\mu -k-\frac{1}{2},\frac{1}{2}}(z)=\frac{z^{\mu -k+1}}{(\mu
-k)(\mu -k+1)}\varphi _{k}(z).  \label{1.8}
\end{equation*}
\end{lemma}

\subsection{Quotients of power series} We will also need the following result (see \cite{biernacki,pv}):

\begin{lemma}\label{lempower}
Consider the power series $f(x)=\displaystyle\sum_{n\geq 0}a_{n}x^n$ and $g(x)=\displaystyle\sum_{n\geq 0}b_{n}x^n$,
where $a_{n}\in \mathbb{R}$ and $b_{n}>0$ for all $n\geq 0$. Suppose that both series converge on $(-r,r)$, for some $r>0$. If the
sequence $\lbrace a_n/b_n\rbrace_{n\geq 0}$ is increasing (decreasing), then the function $x\mapsto{f(x)}/{g(x)}$ is increasing
(decreasing) too on $(0,r)$. The result remains true for the power series
$$f(x)=\displaystyle\sum_{n\geq 0}a_{n}x^{2n}\ \ \ \mbox{and}\ \ \ g(x)=\displaystyle\sum_{n\geq 0}b_{n}x^{2n}.$$
\end{lemma}

\subsection{Zeros of polynomials and entire functions and the Laguerre-P\'olya class} In this subsection
we provide the necessary information about polynomials and entire functions with real zeros. An algebraic polynomial
is called hyperbolic if all its zeros are real.

The simple statement that two real polynomials $p$ and $q$ posses real and interlacing zeros if and only if any linear combinations
of $p$ and $q$ is a hyperbolic polynomial is sometimes called Obrechkoff's theorem. We formulate the following specific statement
that we shall need.
\begin{lemma} \label{OLem}
Let $p(x)=1-a_1 x +a_2 x^2 -a_3 x^3 + \cdots +(-1)^n a_n x^n = (1-x/x_1)\cdots (1-x/x_n)$ be a hyperbolic polynomial with positive zeros
$0< x_1\leq x_2 \leq \cdots \leq x_n$, and normalized by $p(0)=1$. Then, for any constant $C$, the polynomial $q(x) = C p(x) - x\, p'(x)$ is hyperbolic. Moreover, the smallest
zero $\eta_1$   belongs to the interval $(0,x_1)$ if and only if $C<0$.
\end{lemma}

The proof is straightforward; it suffices to apply Rolle's theorem and then count the sign changes of the linear combination at the zeros of $p$. We refer to \cite{BDR, DMR} for further results on monotonicity and asymptotics of zeros of linear combinations of hyperbolic polynomials.

A real entire function $\psi$ belongs to the Laguerre-P\'{o}lya class $\lp$ if it can be represented in the form
$$
\psi(x) = c x^{m} e^{-a x^{2} + \beta x} \prod_{k\geq1}
\left(1+\frac{x}{x_{k}}\right) e^{-\frac{x}{x_{k}}},
$$
with $c,$ $\beta,$ $x_{k} \in \mathbb{R},$ $a \geq 0,$ $m\in
\mathbb{N} \cup\{0\},$ $\sum x_{k}^{-2} < \infty.$
Similarly, $\phi$  is  said to be of
type I in the Laguerre-P\'{o}lya class, written $\varphi \in \lpi$,
if $\phi(x)$ or $\phi(-x)$ can be represented as
$$
\phi(x) = c x^{m} e^{\sigma x} \prod_{k\geq1}\left(1+\frac{x}{x_{k}}\right),
$$
with $c \in \mathbb{R},$ $\sigma \geq 0,$ $m \in
\mathbb{N}\cup\{0\},$ $x_{k}>0,$ $\sum 1/x_{k} < \infty.$
The class $\lp$ is the complement of the space of hyperbolic
polynomials in the topology induced by the uniform convergence
on the compact sets of the complex plane while $\lpi$ is the complement
of the hyperbolic polynomials whose zeros posses a preassigned constant sign.
Given an entire function $\varphi$ with the Maclaurin expansion
$$\varphi(x) = \sum_{k\geq0}\gamma_{k} \frac{x^{k}}{k!},$$
its Jensen polynomials are defined by
$$
g_n(\varphi;x) = g_{n}(x) = \sum_{j=0}^{n} {n\choose j} \g_{j}
x^j.
$$
Jensen  proved the following relation in \cite{Jen12}:
\begin{THEO}\label{JTh}
The function $\varphi$ belongs to $\lp$ ($\lpi$, respectively) if and only if
all the polynomials $g_n(\varphi;x)$, $n=1,2,\ldots$, are hyperbolic (hyperbolic
with zeros of equal sign).
Moreover, the sequence $g_n(\varphi;z/n)$ converges locally
uniformly to $\varphi(z)$.
\end{THEO}
Further information about the Laguerre-P\'olya class can be found in
\cite{Obr, RS} while \cite{DC} contains references  and additional facts
about the Jensen polynomials in general and also about those related to the
Bessel function.

A special emphasis has been given on the question of characterizing the
kernels whose Fourier transform belongs to $\lp$ (see \cite{DR}). The following is a
typical result of this nature, due to P\'olya \cite{pol}.

\begin{THEO} \label{PTh}
\label{pol} Suppose that the function $K$ is positive, strictly increasing
and continuous on $[0, 1)$ and integrable there.
Then the entire functions
\begin{equation*}
U(z)=\int_{0}^{1}K(t) \sin (zt)dt\ \ \ \mbox{and} \ \ \
V(z)=\int_{0}^{1}K(t)\cos(zt)dt
\end{equation*}
have only real and simple zeros and their zeros interlace.
\end{THEO}

In other words, the latter result states that both the sine and the cosine transforms of a kernel
are in the Laguerre-P\'olya class provided the kernel is compactly supported and increasing in the support.

\begin{theorem}\label{ThZ} Let $\mu\in(-1,1),$ $\mu\neq0,$ and $c$ be a constant such that $c<{\mu}+\frac{1}{2}$.  Then the functions  $z\mapsto z s_{\mu -\frac{1}{2},\frac{1}{2}}'(z)-c s_{\mu -\frac{1}{2},\frac{1}{2}}(z)$
can be represented in the form
\begin{equation}
\mu (\mu+1) \left( z\, s_{\mu -\frac{1}{2},\frac{1}{2}}'(z)- c\, s_{\mu -\frac{1}{2},\frac{1}{2}}(z) \right) = z^{\mu+\frac{1}{2}} \psi_\mu(z),  \label{psi}
\end{equation}%
where $\psi_\mu$ is an even entire function and $ \psi_\mu \in \lp$. Moreover,  the smallest positive zero of $\psi_\mu$
does not exceed the first positive zero  of $s_{\mu-\frac{1}{2},\frac{1}{2}}$.

Similarly, if $|\nu |<\frac{1}{2}$ and $d$ is a constant satisfying $d<\nu+1$, then
\begin{equation}
{\frac{\sqrt{\pi}}{2}}\, \Gamma\left(\nu+\frac{3}{2}\right)\ \left(\, z \mathbf{H}_{\nu }'(z)- d \mathbf{H}_{\nu }(z) \, \right) = \left(\frac{z}{2}\right)^{\nu+1}\, \phi_\nu(z),
\label{phinu}
\end{equation}
where $\phi_\nu$ is an entire function in the Laguerre-P\'olya class and the smallest positive zero of $\phi_\nu$
does not exceed the first positive zero  of $\mathbf{H}_{\nu}$.
\end{theorem}

\begin{proof}
First suppose that $\mu\in(0,1).$ Since, by (\ref{LomHypG}),
$$
\mu (\mu+1) s_{\mu -\frac{1}{2},\frac{1}{2}}(z) = \sum_{k\geq0} \frac{(-1)^kz^{2k+\mu+\frac{1}{2}}}{2^{2k}\left(\frac{\mu+2}{2}\right)_k \left(\frac{\mu+3}{2}\right)_k},
$$
then
$$
\mu (\mu+1)  z\, s_{\mu -\frac{1}{2},\frac{1}{2}}'(z) = \sum_{k\geq0} \frac{(-1)^k\left(2k+\mu+\frac{1}{2}\right)z^{2k+\mu+\frac{1}{2}}}{2^{2k}\left(\frac{\mu+2}{2}\right)_k \left(\frac{\mu+3}{2}\right)_k}.
$$
Therefore, (\ref{psi}) holds with
$$
\psi_\mu(z) = \sum_{k\geq0} \frac{2k+{\mu}+\frac{1}{2} -c}{\left(\frac{\mu+2}{2}\right)_k \left(\frac{\mu+3}{2}\right)_k} \left( -\frac{z^2}{4} \right)^{k}.
$$
On the other hand, by Lemma 1,
$$
\mu (\mu+1) s_{\mu -\frac{1}{2},\frac{1}{2}}(z) = z^{\mu+\frac{1}{2}} \varphi_0(z),
$$
and, by \cite[Lemma 3]{Bar2},  we have
\begin{equation}
z\varphi _{0}(z)=\mu (\mu +1)\int_{0}^{1}(1-t)^{\mu -1}\sin (zt)dt, \ \ \mathrm{for}\ \mu >0. \label{integ}
\end{equation}%
Therefore $\varphi_0$ has the Maclaurin expansion
$$
\varphi_0(z) = \sum_{k\geq0}\frac{1}{\left(\frac{\mu+2}{2}\right)_k \left(\frac{\mu+3}{2}\right)_k}  \left( -\frac{z^2}{4} \right)^{k}.
$$
Moreover, (\ref{integ}) and Theorem \ref{PTh} imply that $\varphi_0 \in \lp$ for $\mu\in(0,1)$, so that
the function $\tilde{\varphi}_0(z):= \varphi_0(2\sqrt{z})$,
$$
\tilde{\varphi}_0(\zeta) = \sum_{k\geq0} \frac{1}{\left(\frac{\mu+2}{2}\right)_k \left(\frac{\mu+3}{2}\right)_k}  \left( -\zeta\right)^{k},
$$
belongs to $\lpi$.
Then it follows form Theorem \ref{JTh} that its Jensen polynomials
$$
g_n(\tilde{\varphi}_0;\zeta) = \sum_{k=0}^n  {n\choose k} \frac{k!}{\left(\frac{\mu+2}{2}\right)_k \left(\frac{\mu+3}{2}\right)_k}  \left( -\zeta \right)^{k}
$$
are all hyperbolic. However, observe that the Jensen polynomials of $\tilde{\psi}_\mu(z):= \psi_\mu(2\sqrt{z})$ are simply
$$
-\frac{1}{2}g_n(\tilde{\psi}_\mu;\zeta) =  -\frac{1}{2}\left({\mu}+\frac{1}{2}-c\right)\, g_n(\tilde{\varphi}_0;\zeta) - \zeta\, g_n'(\tilde{\varphi}_0;\zeta).
$$
Lemma \ref{OLem} implies that all zeros of $g_n(\tilde{\psi}_\mu;\zeta)$ are real and positive and that the smallest one
precedes the first zero of  $g_n(\tilde{\varphi}_0;\zeta)$. In view of Theorem \ref{JTh}, the latter conclusion immediately yields that $\tilde{\psi}_\mu \in \lpi$
and that its first zero precedes the one of  $\tilde{\varphi}_0$.  Finally, the first statement of the theorem for $\mu\in(0,1)$ follows after we go back
from  $\tilde{\psi}_\mu$ and $\tilde{\varphi}_0$ to $\psi_\mu$ and $\varphi_0$ by setting $\zeta=-\frac{z^2}{4}$.

Now we prove \eqref{psi} for the case when $\mu\in(-1,0)$. Observe that for $\mu\in(0,1)$ the function \cite[Lemma 3]{Bar2} $$\varphi_1(z)=\sum_{k\geq0}\frac{1}{\left(\frac{\mu+1}{2}\right)_k \left(\frac{\mu+2}{2}\right)_k}  \left( -\frac{z^2}{4} \right)^{k}=\mu \int_{0}^{1}(1-t)^{\mu -1}\cos (zt)dt$$
belongs also to Laguerre-P\'olya class $\mathcal{LP},$ and hence the Jensen polynomials of $\tilde{\varphi}_1(z):= \varphi_1(2\sqrt{z})$ are hyperbolic. Straightforward calculations show that the Jensen polynomials of $\tilde{\psi}_{\mu-1}(z):= \psi_{\mu-1}(2\sqrt{z})$ are
$$-\frac{1}{2}g_n(\tilde{\psi}_{\mu-1};\zeta) =  -\frac{1}{2}\left({\mu}-\frac{1}{2}-c\right)\, g_n(\tilde{\varphi}_1;\zeta) - \zeta\, g_n'(\tilde{\varphi}_1;\zeta).$$
Lemma \ref{OLem} implies that for $\mu\in(0,1)$ all zeros of $g_n(\tilde{\psi}_{\mu-1};\zeta)$ are real and positive and that the smallest one
precedes the first zero of  $g_n(\tilde{\varphi}_1;\zeta)$. This fact, together with Theorem \ref{JTh}, yields that $\tilde{\psi}_{\mu-1} \in \lpi$
and that its first zero precedes the one of  $\tilde{\varphi}_1$. Consequently, the first statement of the theorem for $\mu\in(-1,0)$ follows after we go back
from  $\tilde{\psi}_{\mu-1}$ and $\tilde{\varphi}_1$ to $\psi_{\mu-1}$ and $\varphi_1$ by setting $\zeta=-\frac{z^2}{4}$ and substituting $\mu$ by $\mu+1$.

In order to prove the corresponding statement for (\ref{phinu}), we recall first that the hypergeometric representation (\ref{SrtHypG}) of the Struve function
is equivalent to
$$
\frac{\sqrt{\pi}}{2}\, \Gamma\left(\nu+\frac{3}{2}\right)\,  \mathbf{H}_{\nu}(z) = \sum_{k\geq0} \frac{(-1)^k}{\left(\frac{3}{2}\right)_k \left(\nu+\frac{3}{2}\right)_k} \left( \frac{z}{2} \right)^{2k+\nu+1},
$$
which immediately yields
$$
  \phi_\nu(z) = \sum_{k\geq0} \frac{2k+\nu+1-d}{\left(\frac{3}{2}\right)_k \left(\nu+\frac{3}{2}\right)_k}\left( -\frac{z^2}{4} \right)^{k}.
$$
On the other hand, the integral representation
\begin{equation*}
\mathbf{H}_{\nu }(z)=\frac{2\left(\frac{z}{2}\right) ^{\nu }}{\sqrt{\pi }%
\Gamma \left( \nu +\frac{1}{2}\right) }\int_{0}^{1}(1-t^{2})^{\nu -\frac{1}{2%
}}\sin (zt)dt,
\end{equation*}%
which holds for $\nu >-\frac{1}{2},$ and Theorem \ref{PTh} imply that the even entire function
$$
\mathcal{H}_\nu(z) = \sum_{k\geq0} \frac{1}{\left(\frac{3}{2}\right)_k \left(\nu+\frac{3}{2}\right)_k} \left( -\frac{z^2}{4} \right)^{k}
$$
belongs to the Laguerre-P\'olya class when $|\nu|<\frac{1}{2}$. Then the functions $\tilde{\mathcal{H}}_\nu(z):= \mathcal{H}_{\nu}(2\sqrt{z})$,
$$
\tilde{\mathcal{H}}_\nu(\zeta) = \sum_{k\geq0} \frac{1}{\left(\frac{3}{2}\right)_k \left(\nu+\frac{3}{2}\right)_k} \left(-\zeta \right)^{k},
$$
is in $\lpi$. Therefore, its Jensen polynomials
$$
g_n(\tilde{\mathcal{H}}_\nu;\zeta) = \sum_{k=0}^n  {n\choose k} \frac{k!}{\left(\frac{3}{2}\right)_k \left(\nu+\frac{3}{2}\right)_k} \left( -\zeta \right)^{k}
$$
are hyperbolic, with positive zeros. Then, by Lemma \ref{OLem}, the polynomial
$
-\frac{1}{2}\left({\nu}+{1}-d\right)\, g_n(\tilde{\mathcal{H}}_\nu;\zeta) - \zeta\, g_n'(\tilde{\mathcal{H}}_\nu;\zeta)
$
possesses only real positive zeros. Obviously the latter polynomial coincides with the $n$th Jensen polynomials of
$\tilde{\phi}_\nu(z) = \phi_\nu(2\sqrt{z})$, that is
$$
-\frac{1}{2}g_n(\tilde{\phi}_\nu;\zeta) = -\frac{1}{2}\left({\nu}+{1}-d\right)\, g_n(\tilde{\mathcal{H}}_\nu;\zeta) - \zeta\, g_n'(\tilde{\mathcal{H}}_\nu;\zeta).
$$
Moreover, the smallest zero of $g_n(\tilde{\phi}_\nu;\zeta)$ precedes the first positive zero of   $g_n(\tilde{\mathcal{H}}_\nu;\zeta)$.
This implies that $\phi_\nu \in \lp$ and that its first positive zero is smaller that the one of  $\mathcal{H}_\nu$.
\end{proof}

\section{Proofs of the main results}
\setcounter{equation}{0}

\begin{proof}[Proof of Theorem \ref{theo1}]
We need to show that for the corresponding values of $\mu$ and $\alpha$ the
inequalities
\begin{equation}
\Re \left( \frac{zf_{\mu }'(z)}{f_{\mu }(z)}\right) >\alpha ,%
\text{ \ \ }\Re \left( \frac{zg_{\mu }'(z)}{g_{\mu }(z)}\right)
>\alpha \text{ \ and \ }\Re \left( \frac{zh_{\mu }'(z)}{h_{\mu
}(z)}\right) >\alpha \text{ \ \ }  \label{2.0}
\end{equation}%
are valid for $z\in \mathbb{D}_{r_{\alpha }^{\ast }(f_{\mu })}$, $z\in
\mathbb{D}_{r_{\alpha }^{\ast }(g_{\mu })}$ and $z\in \mathbb{D}_{r_{\alpha
}^{\ast }(h_{\mu })}$ respectively, and each of the above inequalities does
not hold in larger disks. It follows from (\ref{1.6}) that%
\begin{equation*}
f_{\mu }(z)=f_{\mu -\frac{1}{2},\frac{1}{2}}(z)=\left(\mu (\mu +1)s_{\mu -%
\frac{1}{2},\frac{1}{2}}(z)\right)^{\frac{1}{\mu +\frac{1}{2}}}=z\left(\varphi
_{0}(z)\right)^{\frac{1}{\mu +\frac{1}{2}}},
\end{equation*}%
\begin{equation*}
g_{\mu }(z)=g_{\mu -\frac{1}{2},\frac{1}{2}}(z)=\mu (\mu +1)z^{-\mu +\frac{1%
}{2}}s_{\mu -\frac{1}{2},\frac{1}{2}}(z)=z\varphi _{0}(z),
\end{equation*}%
\begin{equation*}
h_{\mu }(z)=h_{\mu -\frac{1}{2},\frac{1}{2}}(z)=\mu (\mu +1)z^{\frac{3-2\mu
}{4}}s_{\mu -\frac{1}{2},\frac{1}{2}}(\sqrt{z})=z\varphi _{0}(\sqrt{z}),
\end{equation*}%
which in turn imply that
\begin{equation*}
\frac{zf_{\mu }'(z)}{f_{\mu }(z)}=1+\frac{z\varphi _{0}'(z)%
}{(\mu +\frac{1}{2})\varphi _{0}(z)}=1-\frac{1}{\mu +\frac{1}{2}}%
\sum\limits_{n\geq 1}\frac{2z^{2}}{z_{\mu ,0,n}^{2}-z^{2}},
\end{equation*}%
\begin{equation*}
\frac{zg_{\mu }'(z)}{g_{\mu }(z)}=1+\frac{z\varphi _{0}'(z)%
}{\varphi _{0}(z)}=1-\sum\limits_{n\geq 1}\frac{2z^{2}}{z_{\mu
,0,n}^{2}-z^{2}},
\end{equation*}%
\begin{equation*}
\frac{zh_{\mu }'(z)}{h_{\mu }(z)}=1+\frac{1}{2}\frac{\sqrt{z}%
\varphi _{0}'(\sqrt{z})}{\varphi _{0}(\sqrt{z})}=1-\sum\limits_{n%
\geq 1}\frac{z}{z_{\mu ,0,n}^{2}-z},
\end{equation*}%
respectively. We note that for $\mu \in (0,1)$ the function $\varphi _{0}$
has only real and simple zeros (see \cite{Bar2}). For $\mu \in (0,1),$ and ${%
n\in }\left\{ 1,2,\dots \right\} $ let $\xi _{\mu ,n}=z_{\mu ,0,n}$ be the $%
n $th positive zero of $\varphi _{0}.$ We know that (see \cite[Lemma 2.1]%
{Kou}) $\xi _{\mu ,n}\in (n\pi ,(n+1)\pi )$ for all $\mu \in (0,1)$ and
${n\in }\left\{ 1,2,\dots \right\}$, which implies that $\xi _{\mu ,n}>\xi
_{\mu ,1}>\pi >1$ for all $\mu \in (0,1)$ and $n \geq 2$. On the other hand,
it is known that \cite{sz} if ${z\in \mathbb{C}}$ and $\beta $ ${\in \mathbb{R}}$ are such that $\beta >{\left\vert
z\right\vert }$, then%
\begin{equation}
\frac{{\left\vert z\right\vert }}{\beta -{\left\vert z\right\vert }}\geq
\Re\left( \frac{z}{\beta -z}\right) .  \label{2.5}
\end{equation}%
Then the inequality
$$
\frac{{\left\vert z\right\vert }^{2}}{\xi _{\mu ,n}^{2}-{\left\vert
z\right\vert }^{2}}\geq\Re\left( \frac{z^{2}}{\xi _{\mu ,n}^{2}-z^{2}}%
\right),
$$
holds get for every $\mu \in (0,1)$, $n\in \mathbb{N}$ and ${\left\vert z\right\vert <}\xi _{\mu ,1}$. Therefore,
\begin{equation*}
\Re\left(\frac{zf_{\mu }'(z)}{f_{\mu }(z)}\right)=1-\frac{1%
}{\mu +\frac{1}{2}}\Re\left(\sum\limits_{n\geq 1}\frac{2z^{2}}{\xi
_{\mu ,n}^{2}-z^{2}}\right) \geq 1-\frac{1}{\mu +\frac{1}{2}}%
\sum\limits_{n\geq 1}\frac{2\left\vert z\right\vert ^{2}}{\xi _{\mu
,n}^{2}-\left\vert z\right\vert ^{2}}=\frac{\left\vert z\right\vert f_{\mu
}'(\left\vert z\right\vert )}{f_{\mu }(\left\vert z\right\vert )},
\end{equation*}%
\begin{equation*}
\Re\left( \frac{zg_{\mu }'(z)}{g_{\mu }(z)}\right) =1-\Re\left(\sum\limits_{n\geq 1}\frac{2z^{2}}{\xi _{\mu ,n}^{2}-z^{2}}\right)
\geq 1-\sum\limits_{n\geq 1}\frac{2\left\vert z\right\vert ^{2}}{\xi _{\mu
,n}^{2}-\left\vert z\right\vert ^{2}}=\frac{\left\vert z\right\vert g_{\mu
}'(\left\vert z\right\vert )}{g_{\mu }(\left\vert z\right\vert )}
\end{equation*}%
and%
\begin{equation*}
\Re\left(\frac{zh_{\mu}'(z)}{h_{\mu }(z)}\right)=1-\Re\left(\sum\limits_{n\geq 1}\frac{z}{\xi _{\mu ,n}^{2}-z}\right) \geq
1-\sum\limits_{n\geq 1}\frac{\left\vert z\right\vert }{\xi _{\mu
,n}^{2}-\left\vert z\right\vert }=\frac{\left\vert z\right\vert h_{\mu
}'(\left\vert z\right\vert )}{h_{\mu }(\left\vert z\right\vert )},
\end{equation*}%
where equalities are attained only when $z=\left\vert z\right\vert =r$. The latter  inequalities and
the minimum principle for harmonic functions imply that the
corresponding inequalities in (\ref{2.0}) hold if and only if $%
\left\vert z\right\vert <x_{\mu ,\alpha },$ $\left\vert z\right\vert <y_{\mu
,\alpha }$ and $\left\vert z\right\vert <t_{\mu ,\alpha },$ respectively,
where $x_{\mu ,\alpha }$, $y_{\mu ,\alpha }$ and $t_{\mu ,\alpha }$ are the
smallest positive roots of the equations%
\begin{equation*}
rf_{\mu }'(r)/f_{\mu }(r)=\alpha ,\text{ \ }rg_{\mu }'(r)/g_{\mu }(r)=\alpha ,\ rh_{\mu }'(r)/h_{\mu }(r)=\alpha.
\end{equation*}%
Since their solutions coincide with the zeros of the functions
$$r\mapsto rs_{\mu -\frac{1}{2},\frac{1}{2}}'(r)-\alpha \left( \mu +\frac{1}{2}%
\right) s_{\mu -\frac{1}{2},\frac{1}{2}}(r),\ r\mapsto rs_{\mu -\frac{1}{2},\frac{1}{2}}'(r)-\left( \mu +\alpha -\frac{1}{2%
}\right) s_{\mu -\frac{1}{2},\frac{1}{2}}(r),
$$
$$
r\mapsto rs_{\mu -\frac{1}{2},\frac{1}{2}}'(r)-\left( \mu +2\alpha -\frac{3}{%
2}\right) s_{\mu -\frac{1}{2},\frac{1}{2}}(r),
$$
the result we need follows from Theorem \ref{ThZ}. In other words, Theorem \ref{ThZ} show that all the zeros of the above three functions are real and their first positive zeros do not exceed the first positive zero $\xi_{\mu,1}$. This guarantees that the above inequalities hold. This completes the proof our theorem when $\mu\in(0,1)$.

Now we prove that the inequalities in \eqref{2.0} also hold for $\mu\in\left(-1,0\right),$ except the first one, which is valid for $\mu\in\left(-\frac{1}{2},0\right).$ In order to do this, suppose that $\mu\in(0,1)$ and adapt the above proof, substituting $\mu$ by $\mu-1$,  $\varphi_0$ by the function $\varphi_1$ and taking into account that the $n$th positive zero of $\varphi _{1},$ denoted by $\zeta_{\mu ,n}=z_{\mu ,1,n},$ satisfies (see \cite{Bar3}) $\zeta _{\mu ,n}>\zeta _{\mu ,1}>\frac{\pi }{2}%
>1 $ for all $\mu \in (0,1)$ and $n\geq 2$. It is worth mentiontioning that
\begin{equation*}
\Re\left(\frac{zf_{\mu-1}'(z)}{f_{\mu-1}(z)}\right)=1-\frac{1%
}{\mu -\frac{1}{2}}\Re\left(\sum\limits_{n\geq 1}\frac{2z^{2}}{\zeta
_{\mu ,n}^{2}-z^{2}}\right) \geq 1-\frac{1}{\mu -\frac{1}{2}}%
\sum\limits_{n\geq 1}\frac{2\left\vert z\right\vert ^{2}}{\zeta_{\mu
,n}^{2}-\left\vert z\right\vert ^{2}}=\frac{\left\vert z\right\vert f_{\mu-1
}'(\left\vert z\right\vert )}{f_{\mu-1}(\left\vert z\right\vert )},
\end{equation*}%
remains true for $\mu\in\left(\frac{1}{2},1\right)$. In this case we use the minimum principle for harmonic functions to ensure that \eqref{2.0} is valid for $\mu-1$ instead of $\mu.$ Thus, using again Theorem \ref{ThZ} and replacing $\mu$ by $\mu+1$, we obtain the statement of the first part for $\mu\in\left(-\frac{1}{2},0\right)$. For $\mu\in(-1,0)$ the proof of the second and third inequalities in \eqref{2.0} go along similar lines.

To prove the statement for part {\bf a} when $\mu \in \left(-1,-\frac{1}{2}\right)$ we observe that the counterpart of (\ref{2.5}) is
\begin{equation}
\real\left( \frac{z}{\beta -z}\right) \geq \frac{-{\left\vert
z\right\vert }}{\beta +{\left\vert z\right\vert }},  \label{2.10}
\end{equation}%
and it holds for all ${z\in \mathbb{C}}$ and $\beta $ ${\in \mathbb{R}}$ such
that $\beta >{\left\vert z\right\vert }$ (see \cite{sz}). From (\ref%
{2.10}), we obtain the inequality
$$
\Re\left( \frac{z^{2}}{\zeta _{\mu ,n}^{2}-z^{2}}\right) \geq \frac{-{%
\left\vert z\right\vert }^{2}}{\zeta _{\mu ,n}^{2}+{\left\vert z\right\vert }%
^{2}},  \label{2.11}
$$
which holds for all $\mu \in \left(0,\frac{1}{2}\right),$ $n\in \mathbb{N}$
and ${\left\vert z\right\vert <}\zeta _{\mu ,1}$ and it  implies that%
\begin{equation*}
\Re\left(\frac{zf_{\mu-1}'(z)}{f_{\mu-1}(z)}\right) =1-%
\frac{1}{\mu -\frac{1}{2}}\Re\left(\sum\limits_{n\geq1}\frac{%
2z^{2}}{\zeta _{\mu ,n}^{2}-z^{2}}\right) \geq 1+\frac{1}{\mu -\frac{1}{2}}%
\sum\limits_{n\geq1}\frac{2\left\vert z\right\vert ^{2}}{\zeta _{\mu
,n}^{2}+\left\vert z\right\vert ^{2}}=\frac{i\left\vert z\right\vert f_{\mu
-1}'(i\left\vert z\right\vert )}{f_{\mu -1}(i\left\vert
z\right\vert )}.
\end{equation*}%
In this case equality is attained if $z=i\left\vert z\right\vert =ir.$ Moreover, the latter inequality implies that
\begin{equation*}
\Re\left( \frac{zf_{\mu -1}'(z)}{f_{\mu -1}(z)}\right) >\alpha
\end{equation*}
if and only if $\left\vert z\right\vert <q_{\mu ,\alpha }$, where $q_{\mu ,\alpha }$ denotes the smallest positive root of the equation $irf_{\mu
-1}'(\mathrm{i}r)/f_{\mu-1}(\mathrm{i}r)=\alpha,$ which is equivalent to
\begin{equation*}
i r s_{\mu -\frac{3}{2},\frac{1}{2}}'(ir)-\alpha \left(\mu -\frac{1}{2}\right)s_{\mu -\frac{3}{2},\frac{1}{2}}(ir)=0,\text{ for }\mu \in \left(0,\frac{1}{2}\right).
\end{equation*}%
Substituting $\mu$ by $\mu +1,$ we obtain
\begin{equation*}
i r s_{\mu -\frac{1}{2},\frac{1}{2}}'(i r)-\alpha \left(\mu +\frac{1}{2}\right)s_{\mu -\frac{1}{2},\frac{1}{2}}(ir)=0,\text{ for }\mu \in \left(-1,-\frac{1}{2}\right).
\end{equation*}
It follows from Theorem \ref{ThZ} that the first positive zero of $z\mapsto izs_{\mu -\frac{1}{2},\frac{1}{2}}'(iz)-\alpha \left(\mu +\frac{1}{2}\right)s_{\mu -\frac{1}{2},\frac{1}{2}}(iz)$ does not exceed $\zeta_{\mu,1}$ which guarantees that the above inequalities are valid. All we need to prove is that the above function has actually only one zero in $(0,\infty)$. Observe that, according to Lemma \ref{lempower}, the function
$$
r\mapsto \frac{irs_{\mu -\frac{1}{2},\frac{1}{2}}'(ir)}{s_{\mu -\frac{1}{2},\frac{1}{2}}(ir)}
$$
is increasing on $(0,\infty)$ as a quotient of two power series whose positive coefficients form the increasing ``quotient sequence'' $\left\{2k+\mu+\frac{1}{2}\right\}_{k\geq0}.$ On the other hand, the above function tends to $\mu+\frac{1}{2}$ when $r\to0,$ so that its graph can intersect the horizontal line $y=\alpha\left(\mu+\frac{1}{2}\right)>\mu+\frac{1}{2}$ only once. This completes the proof of part {\bf a} of the theorem when $\mu\in(-1,0)$.
\end{proof}

\begin{proof}[Proof of Theorem 2]
As in the proof of Theorem \ref{theo1} we need show that, for the
corresponding values of $\nu $ and $\alpha $, the inequalities
\begin{equation}
\Re\left( \frac{zu_{\nu }'(z)}{u_{\nu }(z)}\right) >\alpha ,\
\ \Re\left( \frac{zv_{\nu }'(z)}{v_{\nu }(z)}\right) >\alpha
\ \text{and\ }\Re\left( \frac{zw_{\nu }'(z)}{w_{\nu }(z)}%
\right) >\alpha \ \   \label{strv1}
\end{equation}%
are valid for $z\in \mathbb{D}_{r_{\alpha }^{\ast }(u_{\nu })}$, $z\in
\mathbb{D}_{r_{\alpha }^{\ast }(v_{\nu })}$ and $z\in \mathbb{D}_{r_{\alpha
}^{\ast }(w_{\nu })}$ respectively, and each of the above inequalities does
not hold in any larger disk.

If $\left\vert \nu \right\vert \leq \frac{1}{2},$ then (see \cite[Lemma 1]{BPS}) the Hadamard
factorization of the transcendental entire function $\mathcal{H}_{\nu }$, defined by
\begin{equation*}
\mathcal{H}_{\nu }(z)=\sqrt{\pi }2^{\nu }z^{-\nu -1}\Gamma \left( \nu +\frac{%
3}{2}\right) \mathbf{H}_{\nu }(z),
\end{equation*}
reads as follows%
\begin{equation*}
\mathcal{H}_{\nu }(z)=\prod\limits_{n\geq 1}\left( 1-\frac{z^{2}}{h_{\nu
,n}^{2}}\right) ,
\end{equation*}%
which implies that
\begin{equation*}
\mathbf{H}_{\nu }(z)=\frac{z^{\nu +1}}{\sqrt{\pi }2^{\nu }\Gamma \left( \nu +%
\frac{3}{2}\right) }\prod\limits_{n\geq 1}\left( 1-\frac{z^{2}}{h_{\nu
,n}^{2}}\right),
\end{equation*}%
where $h_{\nu ,n}$ stands for the $n$th positive zero of the Struve function
$\mathbf{H}_{\nu }.$

We know that (see \cite[Theorem 2]{Bar3}) $h_{\nu ,n}>h_{\nu ,1}>1$ for all $%
\left\vert \nu \right\vert \leq \frac{1}{2}$ and $n\in \mathbb{N}$. If $\left\vert \nu \right\vert \leq \frac{1}{2}$ and $\left\vert
z\right\vert <h_{\nu ,1}$, then (\ref{2.5}) imples%
\begin{equation*}
\Re\left(\frac{zu_{\nu }'(z)}{u_{\nu }(z)}\right) =1-\frac{1%
}{\nu +1}\Re\left(\sum\limits_{n\geq 1}\frac{2z^{2}}{h_{\nu
,n}^{2}-z^{2}}\right) \geq 1-\frac{1}{\nu +1}\sum\limits_{n\geq 1}\frac{%
2\left\vert z\right\vert ^{2}}{h_{\nu ,n}^{2}-\left\vert z\right\vert ^{2}}=%
\frac{\left\vert z\right\vert u_{\nu }'(\left\vert z\right\vert )}{%
u_{\nu }(\left\vert z\right\vert )},
\end{equation*}%
\begin{equation*}
\Re\left(\frac{zv_{\nu }'(z)}{v_{\nu }(z)}\right)=1-\Re\left(\sum\limits_{n\geq 1}\frac{2z^{2}}{h_{\nu ,n}^{2}-z^{2}}\right) \geq
1-\sum\limits_{n\geq 1}\frac{2\left\vert z\right\vert ^{2}}{h_{\nu
,n}^{2}-\left\vert z\right\vert ^{2}}=\frac{\left\vert z\right\vert v_{\nu
}'(\left\vert z\right\vert )}{v_{\nu }(\left\vert z\right\vert )}
\end{equation*}%
and%
\begin{equation*}
\Re\left(\frac{zw_{\nu }'(z)}{w_{\nu }(z)}\right)=1-\Re\left(\sum\limits_{n\geq 1}\frac{z}{h_{\nu ,n}^{2}-z}\right)\geq
1-\sum\limits_{n\geq 1}\frac{\left\vert z\right\vert }{h_{\nu
,n}^{2}-\left\vert z\right\vert }=\frac{\left\vert z\right\vert w_{\nu}'(\left\vert z\right\vert )}{w_{\nu }(\left\vert z\right\vert )},
\end{equation*}%
where equalities are attained when $z=\left\vert z\right\vert =r.$ Then minimum principle for
harmonic functions implies that the
corresponding inequalities in (\ref{strv1}) hold if and only if
$\left\vert z\right\vert <\delta _{\nu ,\alpha },$ $\left\vert z\right\vert
<\rho _{\nu ,\alpha }$ and $\left\vert z\right\vert <\sigma _{\nu ,\alpha },$
respectively, where $\delta _{\nu ,\alpha }$, $\rho _{\nu ,\alpha }$ and $%
\sigma _{\nu ,\alpha }$ are the smallest positive roots of the equations%
\begin{equation*}
ru_{\nu }'(r)/u_{\nu }(r)=\alpha ,\text{ \ }rv_{\nu }'(r)/v_{\nu }(r)=\alpha ,\ rw_{\nu }'(r)/w_{\nu }(r)=\alpha.
\end{equation*}%
The solutions of these equations are the zeros of the functions
$$r\mapsto r\mathbf{H}_{\nu }'(r)-\alpha (\nu +1)\mathbf{H}_{\nu }(r),\ r\mapsto r\mathbf{H}_{\nu }'(r)-(\alpha +\nu )\mathbf{H}_{\nu }(r),\ r\mapsto r\mathbf{H}_{\nu }'(r)-(2\alpha +\nu -1)\mathbf{H}_{\nu }(r),$$
which, in view of Theorem \ref{ThZ}, have only real zeros and the smallest positive zero of each of them does not exceed the first
positive zeros of $\mathbf{H}_{\nu }$.
\end{proof}

\end{document}